\documentclass[11pt]{article}

\usepackage[utf8]{inputenc}
\usepackage[T1]{fontenc}
\usepackage{lmodern}
\usepackage[margin=1.13in]{geometry}
\usepackage{microtype}
\usepackage{amsmath,amssymb,amsthm,mathtools}
\usepackage{enumitem}
\usepackage{xcolor}
\usepackage[colorlinks=true,linkcolor=blue!50!black,citecolor=blue!50!black,urlcolor=blue!50!black]{hyperref}

\numberwithin{equation}{section}

\newtheorem{theorem}{Theorem}
\newtheorem{lemma}[theorem]{Lemma}

\theoremstyle{definition}

\theoremstyle{remark}

\newcommand{\Fn}{\mathcal F_n}

\newcommand{\1}{\mathbf 1}

\title{Optimality of Wouter van Doorn's Upper Bound for the Mayer--Erdős Farey Problem}
\author{Ricky Cipollini}
\date{July 25, 2026}

\begin{document}
\maketitle

\begin{abstract}
Let $\Fn$ be the Farey sequence of order $n$, written in increasing order.  Call two fractions
\[
  \frac ab<\frac cd
\]
badly ordered if $a<c$ and $b>d$.  Let $f(n)$ be the minimum number of Farey fractions strictly between two badly ordered fractions in $\Fn$.  We prove
\[
  f(n)=\left(\frac14+o(1)\right)n.
\]
In the equivalent indexing convention of Erdős Problem 1005, this determines the requested asymptotic constant as $c=1/4$.  The upper bound $f(n)\le n/4+O(1)$ was first obtained by Wouter van Doorn; the main result here is the matching lower bound.
\end{abstract}

\section{Introduction}

Let $\Fn$ denote the Farey sequence of order $n$, namely the increasing sequence of reduced fractions $a/b\in[0,1]$ with $b\le n$.  Following the language of similarly ordered Farey fractions, two fractions
\[
  \frac ab<\frac cd
\]
are \emph{badly ordered} if their numerators increase while their denominators decrease, that is,
\[
  a<c,\qquad b>d.
\]
Equivalently, the pair is not similarly ordered in the sense of the Mayer--Erdős problem.

We study the quantity
\[
  f(n)=\min\#\left\{\text{Farey fractions of }\Fn\text{ strictly between a badly ordered pair}\right\}.
\]
The problem is equivalent, up to this intervening-fractions convention, to Erdős Problem 1005, which asks whether the largest guaranteed range of similarly ordered Farey fractions is asymptotic to $cn$ for some positive constant $c$ \cite{ErdosProblem1005}.  Mayer initiated this question, and Erdős proved a linear lower bound \cite{Mayer1942,Erdos1943}.  In recent work, van Doorn sharpened the known bounds and proved in particular the upper bound
\[
  f(n)\le \frac n4+O(1),
\]
conjecturing that this upper bound is optimal \cite{vanDoorn2025}.  We prove that conjecture at the level of the asymptotic constant.

\begin{theorem}\label{thm:main}
Let $f(n)$ be the minimum number of Farey fractions strictly between two badly ordered fractions in $\Fn$.  Then
\[
  f(n)=\left(\frac14+o(1)\right)n.
\]
Consequently, Erdős Problem 1005 has asymptotic constant $c=1/4$.
\end{theorem}

The proof is elementary but somewhat delicate.  The lower bound reduces every badly ordered pair to an elementary interval
\[
  \left(\frac ab,\frac{a+1}{b-1}\right),
\]
and then proves, uniformly in $a,b$, that this interval contains at least $n/4-o(n)$ Farey fractions of order $n$.  The key point is a robust one-dimensional increment estimate for a weighted totient sum.  The upper bound is included for completeness; it is essentially the construction of van Doorn.

\section{Reduction to elementary intervals}

Let
\[
  \frac ab<\frac cd,
  \qquad a<c,
  \qquad b>d
\]
be a badly ordered pair in $\Fn$.  Write
\[
  c=a+h,
  \qquad d=b-r,
\]
where $h,r\ge1$.  First $a\ne0$: if $a=0$, then the reduced fraction $a/b$ is $0/1$, so $b=1$, contradicting $b>d\ge1$.  Thus $a\ge1$.

Since $c/d\in[0,1]$, we have $c\le d$.  Hence
\[
  a+h=c\le d=b-r\le b-1,
\]
and therefore
\[
  1\le a\le b-2.
\]
Moreover,
\[
  \frac{a+1}{b-1}\le \frac{a+h}{b-r},
\]
because
\begin{align*}
  (a+h)(b-1)-(a+1)(b-r)
    &=(h-1)(b-1)+(r-1)(a+1)\ge0.
\end{align*}
Thus every badly ordered pair contains the elementary interval
\[
  I_{a,b}=\left(\frac ab,\frac{a+1}{b-1}\right),
\]
where
\[
  1\le a\le b-2,
  \qquad (a,b)=1,
  \qquad b\le n.
\]
It is therefore enough to prove, uniformly for all such $a,b$, that
\[
  N_n(a,b):=\#\bigl(\Fn\cap I_{a,b}\bigr)
  \ge \frac n4-o(n). \tag{2.1}\label{eq:target}
\]

\section{Auxiliary estimates}

We begin with a primitive progression count.  It is the only place where coprimality is counted explicitly.

\begin{lemma}[Primitive progressions]\label{lem:primitive-progression}
Let $(h,s)=1$, and let $e\ge1$.  The integer solutions of
\[
  hq-sp=e
\]
have $q$ in a single residue class modulo $s$.  Moreover, uniformly for real $A<B$,
\[
\#\{q:A<q<B,\\ (p,q)=1,\\ hq-sp=e\}
 =\frac{\varphi(e)}e\frac{B-A}{s}+O(\tau(e)).
\]
The same estimate holds for $sp-hq=e$.
\end{lemma}

\begin{proof}
Since $(h,s)=1$, the congruence $hq\equiv e\pmod s$ puts $q$ in a single residue class modulo $s$.  By Möbius inversion,
\[
  \1_{(p,q)=1}=\sum_{d\mid(p,q)}\mu(d).
\]
If $d\mid p,q$, then $d\mid e$.  Write $p=dP$ and $q=dQ$.  Then
\[
  hQ-sP=e/d.
\]
Again $Q$ lies in a single residue class modulo $s$.  Hence the number of such $Q$ satisfying $A<dQ<B$ is
\[
  \frac{B-A}{ds}+O(1).
\]
Summing over $d\mid e$ gives
\begin{align*}
\sum_{d\mid e}\mu(d)\left(\frac{B-A}{ds}+O(1)\right)
  &=\frac{B-A}{s}\sum_{d\mid e}\frac{\mu(d)}d+O(\tau(e))\\
  &=\frac{\varphi(e)}e\frac{B-A}{s}+O(\tau(e)).
\end{align*}
The equation $sp-hq=e$ is identical after changing signs.
\end{proof}

\begin{lemma}[Uniform Farey count]\label{lem:farey-count}
Uniformly for intervals $J\subset[0,1]$,
\[
  \#(\Fn\cap J)=\frac{3}{\pi^2}|J|n^2+O(n\log n).
\]
Endpoint conventions for $J$ affect only the $O(n)$ term.
\end{lemma}

\begin{proof}
For each denominator $q$, the number of integers $p$ with $p/q\in J$ is $|J|q+O(1)$, uniformly in $J$.  By Möbius inversion,
\begin{align*}
\#(\Fn\cap J)
  &=\sum_{q\le n}\sum_{\substack{p/q\in J\\(p,q)=1}}1\\
  &=\sum_{d\le n}\mu(d)\sum_{r\le n/d}\bigl(|J|r+O(1)\bigr)\\
  &=|J|\sum_{d\le n}\mu(d)\frac{\lfloor n/d\rfloor^2}{2}
      +O\left(\sum_{d\le n}\frac nd\right)\\
  &=\frac{|J|n^2}{2}\sum_{d\le n}\frac{\mu(d)}{d^2}+O(n\log n).
\end{align*}
Since
\[
  \sum_{d=1}^{\infty}\frac{\mu(d)}{d^2}=\frac1{\zeta(2)}=\frac6{\pi^2},
  \qquad
  \sum_{d>n}\frac1{d^2}=O(1/n),
\]
the result follows.
\end{proof}

\begin{lemma}[Farey gaps]\label{lem:farey-gaps}
Let
\[
  \frac hs<\frac{h'}{s'}
\]
be consecutive fractions in $\mathcal F_Q$.  Then
\[
  h's-hs'=1,
  \qquad s+s'>Q,
\]
and the gap length is
\[
  \frac{h'}{s'}-\frac hs=\frac1{ss'}.
\]
\end{lemma}

\begin{proof}
If $s+s'\le Q$, then the mediant $(h+h')/(s+s')$ lies strictly between $h/s$ and $h'/s'$, and after reduction its denominator is at most $s+s'\le Q$, a contradiction.  Thus $s+s'>Q$.

Put $\Delta=h's-hs'$.  Since $h/s<h'/s'$, we have $\Delta\ge1$.  Suppose that $\Delta\ge2$.  Let $\mathbf u=(h,s)$ and $\mathbf v=(h',s')$.  The lattice generated by $\mathbf u,\mathbf v$ has index $\Delta$, so the half-open parallelogram
\[
  P=\{\alpha\mathbf u+\beta\mathbf v:0\le \alpha,\beta<1\}
\]
contains exactly $\Delta$ lattice points.  Hence it contains a nonzero lattice point $z$.  Write
\[
  z=\alpha\mathbf u+\beta\mathbf v,
  \qquad 0\le\alpha,\beta<1.
\]
If $\alpha+\beta>1$, replace $z$ by $\mathbf u+\mathbf v-z$.  We obtain a nonzero lattice point
\[
  (p,q)=\alpha\mathbf u+\beta\mathbf v
\]
with $0\le\alpha,\beta\le1$ and $\alpha+\beta\le1$.  This point cannot lie on either edge from the origin, because $\mathbf u$ and $\mathbf v$ are primitive.  Thus $\alpha,\beta>0$, and
\[
  \frac hs<\frac pq<\frac{h'}{s'}.
\]
Moreover,
\[
  q=\alpha s+\beta s'
    \le (\alpha+\beta)\max(s,s')
    \le \max(s,s')\le Q.
\]
After reducing $p/q$, the denominator does not increase, contradicting consecutiveness in $\mathcal F_Q$.  Therefore $\Delta=1$, and the displayed gap formula follows immediately.
\end{proof}

The next lemma is the engine behind the constant $1/4$.

\begin{lemma}[Totient increments]\label{lem:totient-increments}
For $x>0$ define
\[
  S(x)=\sum_{1\le e<x}\left(1-\frac ex\right)\frac{\varphi(e)}e,
\]
and put $S(0)=0$.  If $x\ge1$ and $y\ge1$, then
\[
  S(x+y)-S(x)\ge \frac y4.
\]
If $x\ge0$ and $y\ge2$, then the same inequality holds.  In particular, for every integer $m\ge2$,
\[
  S(m)\ge \frac m4.
\]
\end{lemma}

\begin{proof}
Let
\[
  \Phi(m)=\sum_{j\le m}\varphi(j).
\]
We first prove the elementary lower bounds
\begin{equation}\label{eq:Phi-first}
  \Phi(m)\ge \frac{m(m+1)}4\qquad(m\ge1)
\end{equation}
and
\begin{equation}\label{eq:Phi-second}
  \Phi(m)\ge \frac{(m+1)^2}{4}\qquad(m\ge7).
\end{equation}
Let
\[
  P(m)=\#\{1\le u,v\le m:(u,v)=1\}.
\]
Counting coprime pairs according to $\max(u,v)$ gives $P(m)=2\Phi(m)-1$.  Indeed, $(1,1)$ contributes $1$, and for each $k\ge2$ the coprime pairs with maximum $k$ contribute $2\varphi(k)$.

If $(u,v)>1$, then some prime $p$ divides both $u$ and $v$.  Hence, by the union bound,
\[
  P(m)\ge m^2-\sum_{p\le m}\left\lfloor\frac mp\right\rfloor^2
        \ge m^2\left(1-\sum_p\frac1{p^2}\right).
\]
We use the explicit estimate
\begin{equation}\label{eq:prime-square-sum}
  \sum_p\frac1{p^2}<\frac{459}{1000}.
\end{equation}
For completeness, here is a verification.  Every prime larger than $97$ is an odd integer at least $99$, and therefore
\[
\sum_p\frac1{p^2}
\le \sum_{p\le97}\frac1{p^2}
   +\sum_{j\ge0}\frac1{(99+2j)^2}.
\]
The tail is bounded by
\[
\sum_{j\ge0}\frac1{(99+2j)^2}
\le \frac1{99^2}+\int_0^\infty\frac{dt}{(99+2t)^2}
=\frac1{99^2}+\frac1{198}<\frac6{1000}.
\]
An exact rational summation over
\[
2,3,5,7,11,13,17,19,23,29,31,37,41,43,47,53,59,61,67,71,73,79,83,89,97
\]
gives $\sum_{p\le97}p^{-2}<453/1000$, proving \eqref{eq:prime-square-sum}.  Thus
\[
  P(m)>\frac{541}{1000}m^2.
\]
For $m\ge24$,
\[
  \frac{541}{1000}m^2\ge\frac{(m+1)^2}{2}-1,
\]
because this is equivalent to $41m^2-1000m+500\ge0$.  Hence, for $m\ge24$,
\[
  2\Phi(m)-1=P(m)>\frac{(m+1)^2}{2}-1,
\]
and so $\Phi(m)>(m+1)^2/4$.

The remaining finite values needed for \eqref{eq:Phi-first} and \eqref{eq:Phi-second} are
\[
\begin{array}{c|rrrrrrrrrrrr}
 m&1&2&3&4&5&6&7&8&9&10&11&12\\ \hline
 \Phi(m)&1&2&4&6&10&12&18&22&28&32&42&46
\end{array}
\]
and
\[
\begin{array}{c|rrrrrrrrrrrr}
 m&13&14&15&16&17&18&19&20&21&22&23&24\\ \hline
 \Phi(m)&58&64&72&80&96&102&120&128&140&150&172&180.
\end{array}
\]
They verify both bounds.

For $x\in[m,m+1]$, $m\ge1$, we have
\[
  S(x)=A_m-\frac{\Phi(m)}x,
  \qquad A_m=\sum_{j\le m}\frac{\varphi(j)}j.
\]
This remains correct at integer endpoints, because the newly appearing term has coefficient zero.  Put
\[
  F(x)=S(x)-\frac x4.
\]
On $[m,m+1]$,
\[
  F(x)=A_m-\frac{\Phi(m)}x-\frac x4,
  \qquad
  F'(x)=\frac{\Phi(m)}{x^2}-\frac14,
  \qquad
  F''(x)=-\frac{2\Phi(m)}{x^3}<0.
\]
So $F$ is concave on each unit interval.  At integers $m\ge1$,
\[
  S(m+1)-S(m)=\frac{\Phi(m)}{m(m+1)},
\]
and therefore, by \eqref{eq:Phi-first},
\[
  F(m+1)-F(m)=\frac{\Phi(m)}{m(m+1)}-\frac14\ge0.
\]
Thus the integer values $F(m)$ are nondecreasing for $m\ge1$.  Furthermore, for $m\ge7$, \eqref{eq:Phi-second} implies throughout $x\in[m,m+1]$ that
\[
  F'(x)\ge \frac{\Phi(m)}{(m+1)^2}-\frac14\ge0.
\]
Hence $F$ is nondecreasing on $[7,\infty)$.

It remains to control small intervals.  Direct substitution gives, for $x\in[m,m+1]$ with $m=1,\ldots,6$,
\[
F(x+1)-F(x)=
\begin{cases}
\dfrac{x^2-3x+4}{4x(x+1)}, & m=1,\\[1.2ex]
\dfrac{5x^2-19x+24}{12x(x+1)}, & m=2,\\[1.2ex]
\dfrac{x^2-7x+16}{4x(x+1)}, & m=3,\\[1.2ex]
\dfrac{11x^2-69x+120}{20x(x+1)}, & m=4,\\[1.2ex]
\dfrac{(x-15)(x-8)}{12x(x+1)}, & m=5,\\[1.2ex]
\dfrac{17x^2-151x+336}{28x(x+1)}, & m=6.
\end{cases}
\]
Each numerator is positive on its indicated interval, so
\begin{equation}\label{eq:F-shift}
  F(x+1)\ge F(x)\qquad(x\ge1).
\end{equation}

We next prove that, for every integer $m\ge0$,
\begin{equation}\label{eq:max-bound}
  \max_{x\in[m,m+1]}F(x)\le F(m+2).
\end{equation}
For $m=0$, we have $S(x)=0$ on $[0,1]$, so $F(x)=-x/4$ and the maximum is $F(0)=0=F(2)$.  For $m=1,3,5$, the maximum occurs at the right endpoint, as follows from the displayed formula for $F'$ and the values $\Phi(1)=1$, $\Phi(3)=4$, and $\Phi(5)=10$.  For $m\ge7$, \eqref{eq:max-bound} follows from monotonicity of $F$ on $[7,\infty)$.  The remaining cases are direct:
\[
  \max_{x\in[2,3]}F(x)=\frac32-\sqrt2<F(4)=\frac16,
\]
\[
  \max_{x\in[4,5]}F(x)=\frac83-\sqrt6<F(6)=\frac3{10},
\]
and
\[
  \max_{x\in[6,7]}F(x)=\frac{19}{5}-2\sqrt3<F(8)=\frac{57}{140}.
\]
This proves \eqref{eq:max-bound}.

Now let $z=x+y$.  First suppose $x\ge0$ and $y\ge2$.  Choose $m\ge0$ with $x\in[m,m+1]$.  Then $z\ge x+2\ge m+2$.  Since the integer values $F(k)$ are nondecreasing for $k\ge1$, and since $m+2\ge2$, every integer endpoint at or after $m+2$ has value at least $F(m+2)$.  On each unit interval $F$ is concave, hence it lies above the chord joining its endpoint values.  Therefore every $t\ge m+2$ satisfies $F(t)\ge F(m+2)$.  Consequently
\[
  F(z)\ge F(m+2)\ge F(x),
\]
where the last inequality is \eqref{eq:max-bound}.  Hence
\[
  S(x+y)-S(x)=\frac y4+F(x+y)-F(x)\ge\frac y4.
\]

It remains to treat $x\ge1$ and $1\le y<2$.  Choose $m\ge1$ with $x\in[m,m+1]$.  If $z\ge m+2$, the previous paragraph gives $F(z)\ge F(x)$.  Otherwise
\[
  z\in[x+1,m+2]\subset[m+1,m+2].
\]
Since $F$ is concave on $[m+1,m+2]$, its minimum on $[x+1,m+2]$ is attained at one of the endpoints.  Thus
\[
  F(z)\ge \min\{F(x+1),F(m+2)\}.
\]
By \eqref{eq:F-shift}, $F(x+1)\ge F(x)$, and by \eqref{eq:max-bound}, $F(m+2)\ge F(x)$.  Hence again $F(z)\ge F(x)$, and the claimed increment inequality follows.  Finally, taking $x=0$ and $y=m\ge2$ gives $S(m)\ge m/4$.
\end{proof}

\section{The lower bound}

We now prove \eqref{eq:target}.  Throughout this section $1\le a\le b-2$, $(a,b)=1$, and $b\le n$.

First note that
\begin{equation}\label{eq:I-length}
  |I_{a,b}|=\frac{a+1}{b-1}-\frac ab=\frac{a+b}{b(b-1)}>\frac1b.
\end{equation}
If $b\le n/\log^2 n$, Lemma \ref{lem:farey-count} gives
\[
  N_n(a,b)=\frac3{\pi^2}|I_{a,b}|n^2+O(n\log n)
           \gg \frac{n^2}{b}-O(n\log n)
           \gg n\log^2 n.
\]
Thus $N_n(a,b)\ge n/4-o(n)$ in this range.  Henceforth assume
\begin{equation}\label{eq:b-large}
  b>\frac n{\log^2 n},
\end{equation}
and put
\[
  \mu=\frac nb.
\]
Then
\begin{equation}\label{eq:mu-range}
  1\le \mu<\log^2 n.
\end{equation}

\subsection{When the elementary interval contains a small rational}

Let
\[
  Q=\lfloor b^{2/3}\rfloor.
\]
Suppose that $I_{a,b}$ contains a reduced rational $h/s$ strictly inside it with $s\le Q$.  Since $I_{a,b}\subset(0,1)$, we have $1\le h<s$.  Define
\[
  u=bh-as,
  \qquad
  v=(a+1)s-(b-1)h.
\]
Because $h/s\in I_{a,b}$, both $u$ and $v$ are at least $1$, and
\begin{equation}\label{eq:u-plus-v}
  u+v=bh-as+(a+1)s-(b-1)h=s+h.
\end{equation}

We count fractions on the left of $h/s$.  Write
\[
  e=hq-sp>0.
\]
Then $p=(hq-e)/s$.  The condition $p/q>a/b$ is equivalent to
\[
  b(hq-e)>asq,
\]
or
\[
  uq>be.
\]
Therefore the left side contributes
\[
  \frac bs\sum_{1\le e<\mu u}\left(\mu-\frac eu\right)\frac{\varphi(e)}e
  +O\left(\sum_{e<\mu u}\tau(e)\right).
\]
Similarly, for fractions on the right of $h/s$, write $e=sp-hq>0$.  The upper condition $p/q<(a+1)/(b-1)$ is equivalent to
\[
  vq>(b-1)e.
\]
For a lower bound we impose the stronger condition $vq>be$.  Thus the right side contributes at least
\[
  \frac bs\sum_{1\le e<\mu v}\left(\mu-\frac ev\right)\frac{\varphi(e)}e
  +O\left(\sum_{e<\mu v}\tau(e)\right).
\]
Define
\[
  T_\mu(m)=\sum_{1\le e<\mu m}\left(\mu-\frac em\right)\frac{\varphi(e)}e.
\]
Combining the two sides and using \eqref{eq:u-plus-v},
\begin{equation}\label{eq:caseA-combine}
  N_n(a,b)\ge \frac bs\{T_\mu(u)+T_\mu(v)\}-O(\mu s\log n).
\end{equation}
Because $\mu\ge1$,
\[
  T_\mu(m)\ge \mu S(m).
\]
Indeed, for every $e<m$,
\[
  \mu-\frac em\ge \mu\left(1-\frac em\right),
\]
and the additional terms with $m\le e<\mu m$, if any, are nonnegative.  Hence
\begin{equation}\label{eq:caseA-S}
  N_n(a,b)\ge \frac ns\{S(u)+S(v)\}-O(\mu s\log n).
\end{equation}

We claim that
\begin{equation}\label{eq:Suv}
  S(u)+S(v)\ge \frac s4.
\end{equation}
If $u,v\ge2$, Lemma \ref{lem:totient-increments} gives
\[
  S(u)+S(v)\ge\frac{u+v}{4}=\frac{s+h}{4}\ge\frac s4.
\]
If one of $u,v$ equals $1$, then the other equals $s+h-1$.  Since $1\le h<s$, we have $s\ge2$ and $s+h-1\ge s\ge2$; Lemma \ref{lem:totient-increments} again gives \eqref{eq:Suv}.  From \eqref{eq:caseA-S},
\[
  N_n(a,b)\ge \frac n4-O(\mu s\log n).
\]
Since $s\le Q\le b^{2/3}$,
\[
  \mu s\log n\le \frac nb b^{2/3}\log n=n b^{-1/3}\log n.
\]
Using \eqref{eq:b-large}, this is at most $n^{2/3}\log^{5/3}n=o(n)$.  Therefore, in this case,
\[
  N_n(a,b)\ge \frac n4-o(n).
\]

\subsection{When no small rational lies inside}

Assume now that $I_{a,b}$ contains no reduced rational of denominator at most $Q=\lfloor b^{2/3}\rfloor$ strictly inside it.  Then $I_{a,b}$ lies in the closure of one gap of $\mathcal F_Q$.  Write this gap as
\[
  \frac hs<\frac{h'}{s'}.
\]
By Lemma \ref{lem:farey-gaps},
\[
  h's-hs'=1,
  \qquad s+s'>Q,
  \qquad \frac{h'}{s'}-\frac hs=\frac1{ss'}.
\]
Since $I_{a,b}$ lies inside the closure of the gap,
\[
  \frac1{ss'}\ge |I_{a,b}|>\frac1b.
\]
Thus
\begin{equation}\label{eq:ssprime}
  ss'<b.
\end{equation}
Let
\[
  D_{\min}=\min(s,s'),\qquad D_{\max}=\max(s,s').
\]
Because $D_{\min}D_{\max}<b$ and $D_{\min}+D_{\max}>Q$, we have $D_{\max}>Q/2$, and hence
\begin{equation}\label{eq:Dmin}
  D_{\min}<\frac{2b}{Q}=O(b^{1/3}).
\end{equation}
For sufficiently large $b$, $Q\ge b^{2/3}/2$, so
\begin{equation}\label{eq:Dmax}
  D_{\max}\gg b^{2/3}.
\end{equation}
Choose the endpoint of the Farey gap whose denominator is $D_{\min}$, and call it $h/s$.  Thus
\begin{equation}\label{eq:small-s}
  s=O(b^{1/3}).
\end{equation}
There are two cases, according to whether this small endpoint lies on the right or on the left of $I_{a,b}$.

\subsubsection{The small endpoint lies on the right}

Suppose
\[
  \frac ab<\frac{a+1}{b-1}\le \frac hs.
\]
Define
\[
  r=(b-1)h-(a+1)s\ge0,
  \qquad
  w=s+h.
\]
Then
\begin{equation}\label{eq:rw-right}
  bh-as=r+w.
\end{equation}
Since $0\le h\le s$ and \eqref{eq:small-s} holds,
\[
  w=O(b^{1/3}).
\]
Let the other endpoint of the Farey gap have denominator $s_0$.  By \eqref{eq:Dmax}, $s_0\gg b^{2/3}$.  Moreover,
\[
  0\le \frac hs-\frac{a+1}{b-1}=\frac{r}{(b-1)s}\le \frac1{ss_0},
\]
so
\[
  r\le \frac{b-1}{s_0}=O(b^{1/3}).
\]
Thus
\begin{equation}\label{eq:srw-small-right}
  s,w,r=O(b^{1/3}).
\end{equation}

Let $p/q\in I_{a,b}$.  Then $p/q<h/s$.  Write $e=hq-sp>0$.  The lower condition $p/q>a/b$ is equivalent, using \eqref{eq:rw-right}, to
\[
  q>\frac{be}{r+w}.
\]
If $r>0$, the upper condition $p/q<(a+1)/(b-1)$ is equivalent to
\[
  q<\frac{(b-1)e}{r}.
\]
Therefore Lemma \ref{lem:primitive-progression} gives
\begin{equation}\label{eq:right-Bmub}
  N_n(a,b)
  \ge \frac bs\mathcal B_{\mu,b}(r,w)-O(\mu(r+w)\log n),
\end{equation}
where
\[
  \mathcal B_{\mu,b}(r,w)
  =\sum_{e\ge1}
  \left(
  \min\left(\mu,\left(1-\frac1b\right)\frac er\right)
  -\frac e{r+w}
  \right)_+
  \frac{\varphi(e)}e.
\]
Here $x_+=\max(x,0)$.  A contributing $e$ satisfies $e<\mu(r+w)$, which explains the error term.

Compare this with
\[
  \mathcal B_{\mu}(r,w)
  =\sum_{e\ge1}
  \left(
  \min\left(\mu,\frac er\right)-\frac e{r+w}
  \right)_+
  \frac{\varphi(e)}e.
\]
A direct split at $e=\mu r$ gives the exact identity
\begin{equation}\label{eq:B-identity}
  \mathcal B_\mu(r,w)=\mu\{S(\mu(r+w))-S(\mu r)\}.
\end{equation}
Since $r\ge1$, $\mu r\ge1$; since $w\ge1$, $\mu w\ge1$.  Lemma \ref{lem:totient-increments} therefore gives
\begin{equation}\label{eq:B-lower}
  \mathcal B_\mu(r,w)\ge \frac{\mu^2w}{4}.
\end{equation}
The map $X\mapsto (X-c)_+$ is $1$-Lipschitz.  Replacing $e/r$ by $(1-1/b)e/r$ inside the minimum can decrease the summand by at most
\[
  \frac{e}{br}\cdot\frac{\varphi(e)}e=\frac{\varphi(e)}{br}.
\]
Only $e\le\mu(r+w)$ can contribute.  Hence
\[
  \mathcal B_{\mu,b}(r,w)
  \ge \mathcal B_\mu(r,w)
      -O\left(\frac1{br}\sum_{e\le\mu(r+w)}\varphi(e)\right)
  \ge \frac{\mu^2w}{4}
      -O\left(\frac{\mu^2(r+w)^2}{br}\right).
\]
Substituting in \eqref{eq:right-Bmub},
\begin{equation}\label{eq:right-r-positive}
  N_n(a,b)
  \ge \frac bs\cdot\frac{\mu^2w}{4}
      -O\left(\frac{\mu^2(r+w)^2}{sr}\right)
      -O(\mu(r+w)\log n).
\end{equation}
By \eqref{eq:mu-range} and \eqref{eq:srw-small-right}, the first error is $O(b^{2/3}\log^4 n)=o(n)$, and the second error is
\[
  O\left(\frac nb b^{1/3}\log n\right)=O(nb^{-2/3}\log n)=o(n)
\]
using \eqref{eq:b-large}.  Since $w=s+h\ge s$,
\[
  \frac bs\cdot\frac{\mu^2w}{4}\ge\frac{b\mu^2}{4}=\frac{n^2}{4b}\ge\frac n4.
\]
Therefore $N_n(a,b)\ge n/4-o(n)$ when $r>0$.

If $r=0$, then $(a+1)/(b-1)=h/s$, so the upper cutoff disappears.  The same determinant count gives
\[
  N_n(a,b)
  \ge \frac bs\sum_{e\ge1}\left(\mu-\frac ew\right)_+\frac{\varphi(e)}e
        -O(\mu w\log n).
\]
The sum equals $\mu S(\mu w)$.  Since the right endpoint $h/s$ lies to the right of $I_{a,b}\subset(0,1)$, we have $h\ge1$, hence $w=s+h\ge2$ and $\mu w\ge2$.  By Lemma \ref{lem:totient-increments},
\[
  S(\mu w)\ge\frac{\mu w}{4}.
\]
Thus
\[
  N_n(a,b)\ge \frac bs\cdot\frac{\mu^2w}{4}-O(\mu w\log n)
             \ge \frac n4-o(n).
\]
This completes the case where the small endpoint lies on the right.

\subsubsection{The small endpoint lies on the left}

Now suppose
\[
  \frac hs\le\frac ab<\frac{a+1}{b-1}.
\]
Define
\[
  r=as-bh\ge0,
  \qquad
  w=s+h.
\]
Then
\begin{equation}\label{eq:rw-left}
  (a+1)s-(b-1)h=r+w.
\end{equation}
As above, $s=O(b^{1/3})$, $w=O(b^{1/3})$, and if $s_0$ is the denominator of the other Farey-gap endpoint, then $s_0\gg b^{2/3}$.  Moreover,
\[
  0\le \frac ab-\frac hs=\frac r{bs}\le\frac1{ss_0},
\]
so
\begin{equation}\label{eq:srw-small-left}
  s,w,r=O(b^{1/3}).
\end{equation}

Let $p/q\in I_{a,b}$.  Then $p/q>h/s$.  Write $e=sp-hq>0$.  The upper condition $p/q<(a+1)/(b-1)$ gives, by \eqref{eq:rw-left},
\[
  (r+w)q>(b-1)e.
\]
We impose the stronger condition
\[
  q>\frac{be}{r+w}.
\]
If $r>0$, the lower condition $p/q>a/b$ is equivalent to
\[
  q<\frac{be}{r}.
\]
Therefore
\[
  N_n(a,b)
  \ge \frac bs\mathcal B_\mu(r,w)-O(\mu(r+w)\log n),
\]
where $\mathcal B_\mu$ is the sum in \eqref{eq:B-identity}.  Since $r\ge1$, Lemma \ref{lem:totient-increments} and \eqref{eq:B-identity} give
\[
  \mathcal B_\mu(r,w)\ge\frac{\mu^2w}{4}.
\]
The same error estimates as above, using \eqref{eq:srw-small-left}, yield
\[
  N_n(a,b)\ge\frac bs\cdot\frac{\mu^2w}{4}-o(n)\ge\frac n4-o(n),
\]
because $w\ge s$.

It remains to handle $r=0$.  Then $h/s=a/b$, and the lower cutoff disappears.  The determinant count gives
\[
  N_n(a,b)
  \ge \frac bs\sum_{e\ge1}\left(\mu-\frac ew\right)_+\frac{\varphi(e)}e
        -O(\mu w\log n).
\]
Again the sum is $\mu S(\mu w)$.  If $h=0$ and $s=1$, then $r=as-bh=a\ne0$, because $a\ge1$.  Hence, in the case $r=0$, we must have $h\ge1$, and so $w=s+h\ge2$.  Lemma \ref{lem:totient-increments} gives
\[
  S(\mu w)\ge\frac{\mu w}{4}.
\]
Consequently
\[
  N_n(a,b)\ge\frac bs\cdot\frac{\mu^2w}{4}-o(n)
             \ge\frac n4-o(n).
\]
This completes the no-small-rational case, and therefore proves the uniform lower bound \eqref{eq:target}.

\section{The upper bound}

The following construction gives $f(n)\le n/4+O(1)$.  As mentioned above, this upper bound was first obtained by van Doorn \cite{vanDoorn2025}; the proof is included to make the asymptotic conclusion self-contained.

Let
\[
  m=\left\lfloor\frac n4\right\rfloor
\]
and define
\[
  L=\frac{2m-1}{4m},
  \qquad
  R=\frac{2m}{4m-1}.
\]
Both fractions are reduced, and $4m\le n$, $4m-1\le n$, so $L,R\in\Fn$.  Moreover
\[
  L<R,
  \qquad 2m-1<2m,
  \qquad 4m>4m-1,
\]
so $L,R$ form a badly ordered pair.

We count the fractions $p/q\in\Fn$ satisfying $L<p/q<R$.  Write
\[
  e=q-2p.
\]
If $e=0$, then $p/q=1/2$, which contributes one fraction.

Suppose first that $e>0$.  Then $q=2p+e$, so $p/q<1/2<R$.  The condition $p/q>L$ is
\[
  4mp>(2m-1)q.
\]
Substituting $q=2p+e$ gives
\[
  2p>(2m-1)e.
\]
Also
\[
  q=2p+e\le n=4m+O(1).
\]
For $e=1$, these inequalities imply
\[
  m\le p\le2m+O(1),
\]
and every fraction $p/(2p+1)$ is reduced.  Hence $e=1$ contributes
\[
  m+O(1)=\frac n4+O(1)
\]
fractions.  For $e\ge2$, the inequalities
\[
  2p>(2m-1)e,
  \qquad
  2p+e\le4m+O(1)
\]
leave only $O(1)$ possibilities, since the available interval for $p$ has length $m(2-e)+O(1)$.

Now suppose $e<0$.  Write $e=-s$, $s\ge1$.  Then $q=2p-s$ and $p/q>1/2>L$.  The upper inequality $p/q<R$ is
\[
  (4m-1)p<2mq.
\]
Substituting $q=2p-s$ gives
\[
  p>2ms.
\]
Together with $q=2p-s\le n=4m+O(1)$, this implies
\[
  4ms-s<4m+O(1).
\]
Therefore only $s=1$ can contribute for large $m$, and even then the possible $p$ form an interval of length $O(1)$.  Thus the entire $e<0$ side contributes $O(1)$ fractions.

Consequently
\[
  \#\bigl(\Fn\cap(L,R)\bigr)=m+O(1)=\frac n4+O(1).
\]
Since $L,R$ are badly ordered,
\[
  f(n)\le \frac n4+O(1).
\]

\section{Conclusion}

Section 4 proves the uniform lower bound
\[
  f(n)\ge \frac n4-o(n),
\]
and Section 5 gives the explicit upper bound
\[
  f(n)\le \frac n4+O(1).
\]
Together these imply
\[
  f(n)=\left(\frac14+o(1)\right)n.
\]
Thus van Doorn's upper-bound constant is optimal, and the asymptotic form of Erdős Problem 1005 is resolved with constant $1/4$.

\section*{Acknowledgments and contribution statement}

This paper was written by GPT-5.5 Pro. The mathematical findings and proof strategy are due to Ricky Cipollini together with GPT-5.5 Thinking.  The upper bound $f(n)\le n/4+O(1)$ and the conjecture that this is the optimal constant are due to Wouter van Doorn \cite{vanDoorn2025}; that contribution is used here as the upper-bound half of the final asymptotic. The author also thanks Aristotle and van Doorn for the Lean 4 formalization of the proof. The formalization is available at \url{https://github.com/Woett/Lean-files/blob/main/ErdosProblem1005.lean} and can be type-checked online using \href{https://live.lean-lang.org/#project=mathlib-v4.28.0&url=https%3A%2F%2Fraw.githubusercontent.com%2FWoett%2FLean-files%2Fmain%2FErdosProblem1005.lean}{Lean Live}.


\begin{thebibliography}{9}

\bibitem{Mayer1942}
A.~E. Mayer,
\newblock A mean value theorem concerning Farey series,
\newblock \emph{Quarterly Journal of Mathematics} \textbf{os-13} (1942), no.~1, 48--57.

\bibitem{Erdos1943}
P.~Erdős,
\newblock A note on Farey series,
\newblock \emph{Quarterly Journal of Mathematics} \textbf{os-14} (1943), no.~1, 82--85.

\bibitem{vanDoorn2025}
W.~van Doorn,
\newblock Improved bounds for the Mayer--Erdős phenomenon on similarly ordered Farey fractions,
\newblock arXiv:2509.00121 [math.NT], 2025,
\newblock \url{https://arxiv.org/abs/2509.00121}.

\bibitem{ErdosProblem1005}
T.~F. Bloom,
\newblock Erdős Problem \#1005,
\newblock \url{https://www.erdosproblems.com/1005}.

\end{thebibliography}
\end{document}